\documentclass[11pt, leqno]{article}

\usepackage{verbatim}
\usepackage{hyperref}
\usepackage{amsmath}
\usepackage{enumerate, amsthm}
\usepackage{mathrsfs}
\usepackage{xcolor}
\usepackage[margin=1 in]{geometry}
\usepackage[margin={1.5cm, 1.5cm},font=small, labelsep=endash]{caption}
\usepackage{mathtools}
\usepackage{tikz}
\usetikzlibrary{patterns}
\usetikzlibrary{arrows}
\usepackage{ifthen}
\usepackage{bm}
\usepackage{amssymb}
\usepackage{centernot}

\numberwithin{equation}{section}

\theoremstyle{plain}
\newtheorem{theorem}{Theorem}[section]
\newtheorem{lemma}[theorem]{Lemma}

\newtheorem{proposition}[theorem]{Proposition}

\theoremstyle{definition}
\newtheorem{definition}[theorem]{Definition}
\newtheorem{assumption}[theorem]{Assumption}
\theoremstyle{remark}
\newtheorem{remark}[theorem]{Remark}

\renewcommand{\P}{\mathbb P}
\newcommand{\E}{\mathbb E}
\newcommand{\R}{\mathbb R}
\newcommand{\Z}{\mathbb Z}
\newcommand{\N}{\mathbb N}

\newcommand{\si}[1]{{\mathchoice{}{}{\scriptscriptstyle}{}#1}}
\newcommand{\sN}{{\mathchoice{}{}{\scriptscriptstyle}{}N}}

\newcommand{\lr}[4]{#3\xleftrightarrow[#1]{#2} #4}
\newcommand{\nlr}[4]{#3\centernot{\xleftrightarrow[#1]{#2}} #4}

\DeclareMathAlphabet{\pazocal}{OMS}{zplm}{m}{n}

\title{Sharp phase transition for Gaussian percolation in all dimensions}
\author{Franco Severo\footnotemark[1]\footnote{Universit\'e de Gen\`eve, franco.severo@unige.ch}}

\date{}

\begin{document}
\thispagestyle{empty}
\maketitle

\begin{abstract}
We consider the level-sets of continuous Gaussian fields on $\R^d$ above a certain level $-\ell\in \R$, which defines a percolation model as $\ell$ varies. We assume that the covariance kernel satisfies certain regularity, symmetry and positivity  conditions as well as a polynomial decay with exponent greater than $d$ (in particular, this includes the Bargmann--Fock field). Under these assumptions, we prove that the model undergoes a sharp phase transition around its critical point $\ell_c$. More precisely, we show that connection probabilities decay exponentially for $\ell<\ell_c$ and percolation occurs in sufficiently thick 2D slabs for $\ell>\ell_c$.
This extends results recently obtained in dimension $d=2$ to arbitrary dimensions through completely different techniques.
The result follows from a global comparison with a truncated (i.e.~with finite range of dependence) and discretized (i.e.~defined on the lattice $\varepsilon\Z^d$) version of the model, which may be of independent interest. The proof of this comparison relies on an interpolation scheme that integrates out the long-range and infinitesimal correlations of the model while compensating them with a slight change in the parameter $\ell$.
\end{abstract}

\section{Introduction}\label{sec:intro}

Phase transition phenomena for level-sets of random fields on $\R^d$ were first studied in the 80's by Molchanov and Stepanov \cite{MoSt83a, MoSt83b, MoSt86c} and has been the object of intense research in the last decade -- see e.g.~\cite{BG17, RVb, BM18, MV, MRV20}. One of the major interests in this area of research lies on its strong links with the geometry of nodal sets of random polynomials and random spherical harmonics \cite{Sa1, An16, NaSo09, CaSa19, SaWi19}. It is believed that, under suitable assumptions on a Gaussian field $f$ on $\R^d$, the behavior of its level-sets should be very similar to that of Bernoulli percolation on $\Z^d$, even at criticality \cite{BogSch02,BogSch07}. In particular, one expects to observe a \emph{sharp phase transition} around the corresponding critical level $\ell_c$. For Bernoulli percolation, this corresponds to the exponential decay of cluster size distribution in the subcritical phase -- 
proved independently by Menshikov \cite{Men86} and Aizenman and Barsky \cite{AizBar87} -- and the existence of an infinite cluster on sufficiently thick 2D slabs in the supercritical phase -- 
proved by Grimmett and Marstrand \cite{GriMar90}. 
In the present paper, we prove that the corresponding results hold for the level-sets of continuous Gaussian fields whose covariance kernel satisfies certain regularity, symmetry and positivity conditions as well as a polynomial decay with exponent greater than $d$. 
As an important example, our assumptions are satisfied by the Bargmann--Fock field, for which these results were known only in the planar case $d=2$ \cite{RVb}. 

Our approach is based on an interpolation scheme aimed at integrating out the long-range and infinitesimal correlations of the model, thus comparing it with a truncated and discretized counterpart. This strategy is inspired by a recent proof \cite{DGRS19} of sharpness for the level-sets of the (discrete) Gaussian free field (GFF) on $\Z^d$, $d\geq 3$. The proof in \cite{DGRS19} relies extensively on a decomposition of the GFF as a sum of independent finite-range fields, which is used to perform certain ``local surgeries'' with a low cost. Continuous Gaussian fields on $\R^d$ on the other hand tend to be much more rigid objects and, in particular, such a decomposition is not available in general. The Bargmann--Fock field for instance is analytic almost surely, hence it is determined by its restriction to any open set, which in turn implies that one cannot decompose the field as a sum independent finite-range fields. This rigidity makes the implementation of finite-energy arguments much more challenging in the continuum. In order to overcome this difficulty, we make use of a shift-argument based on the Cameron-Martin formula. We hope that similar shift-arguments can be used as a general way to bypass the lack of finite-energy property in the study of Gaussian percolation.

Beyond the contrast between discrete and continuum setup, we would like to stress that the fields considered in the present paper have faster decay of correlations than the GFF and belong to a different universality class (namely, that of Bernoulli percolation). Another important difference with \cite{DGRS19} is that we do not prove the most classical notion of ``supercritical sharpness'', i.e.~that local uniqueness events happen with high probability in the supercritical phase. Nonetheless, we strongly believe that our theorem is the first step to proving this result for the Gaussian fields considered here. We would also like to highlight that unlike \cite{DGRS19}, our interpolation scheme leads to a global comparison result that is valid throughout the parameter space and does not require the assumption (by contradiction) that sharpness does not hold. We believe that this global comparison result may be of independent interest.

In recent years, a great progress has been made in the study of Gaussian percolation on the plane (i.e.~for $d=2$), for which sharpness of phase transition has been established in a series of works with progressively milder assumptions, see e.g.~\cite{RVb,MV,Riv19,MRV20}. This is mainly due to a special \emph{duality} property that is only available in the planar setting. In particular, one has $\ell_c=0$ on $\R^2$ in great generality \cite{MRV20}, which can be seen as the analogue of Kesten's celebrated result \cite{Kes80} establishing that $p_c=1/2$ for Bernoulli percolation on $\Z^2$. Another tool that makes the study of planar percolation models simpler is the Russo--Seymour--Welsh theory of crossing probabilities. In larger dimensions though, these techniques break down and the study of Gaussian percolation remains rather limited. Subcritical sharpness in dimensions $d\geq3$ has been proved only for fields with finite-range correlations in a very recent work \cite{DM21}, and it appears that no sort of supercritical sharpness result has been obtained in the literature. We therefore hope that the present work will contribute to a better understanding of the off-critical phases of Gaussian percolation, especially in dimensions $d\geq3$.

\subsection{Gaussian percolation and sharp phase transition}

Let $f$ be a stationary, centered, ergodic, continuous Gaussian field on $\R^d$. Such a field is characterized by its covariance kernel
\begin{equation}
	\kappa(x):=\E[f(0)f(x)],~~~ x\in\R^d.
\end{equation}
We are interested in the connectivity properties of the (upper-)excursion set
\begin{equation}
	\pazocal{E}(\ell):=\{x\in\R^d:~ f(x)\geq -\ell\}
\end{equation}
as the parameter $\ell\in \R$ varies. We say that $\pazocal{E}(\ell)$ \emph{percolates} if it contains an unbounded connected component. The percolation \emph{critical level} is then defined as
\begin{equation}\label{eq:l_c_def}
	\ell_c:=\sup\Big\{\ell\in \R:~ \P[\pazocal{E}(\ell) \text{ percolates}]=0\Big\}.
\end{equation}
We are going to make the following assumptions on the covariance kernel $\kappa$.

\begin{assumption}\label{assump:q}
	Suppose that there exists an $L^2$ function $q:\R^d\to\R$ such that 
	\begin{enumerate}[(i)]
		\item $\kappa=q\star q$, where $\star$ denotes the convolution on $\R^d$,
		\item there exist $\beta>d/2$ and $C\in(0,\infty)$ such that  for every $|x|\geq1$,
		\begin{equation}\label{eq:decay_q}
			\max\{|q(x)|, |\triangledown q(x)|\}\leq C|x|^{-\beta},
		\end{equation}
		\item $q\in C^3(\R^d)$ and $\partial^\alpha q\in L^2(\R^d)$ for every multi-index $\alpha$ with $|\alpha|\leq3$,
		\item $q(x)$ is invariant under sign changes and permutation of the coordinates of $x$,
		\item $q\geq 0$.
	\end{enumerate}
\end{assumption}

We denote by $B_R:=[-R,R]^d$ the $\ell^\infty$-ball of radius $R\geq0$ centered at the origin. Given a random subset $\pazocal{L}\subseteq\R^d$ and a (deterministic) domain $D\subseteq \R^d$, we say that $\pazocal{L}$ percolates in $D$ if there is an unbounded connected component in $\pazocal{L}\cap D$. For (deterministic) subsets $A,B\subseteq\R^d$, we will denote by $\{\lr{D}{\pazocal{L}}{A}{B}\}$ the event that $\pazocal{L}\cap D$ contains a path from $A$ to $B$. We also write $\{\lr{D}{\pazocal{L}}{A}{\infty}\}$ for the event that there exists an infinite connected component in $\pazocal{L}\cap D$ intersecting $A$. As for the complementary event $\{\lr{D}{\pazocal{L}}{A}{B}\}^{\mathsf{c}}$, we simply write $\{\nlr{D}{\pazocal{L}}{A}{B}\}$.
We shall drop $D$ from the notation whenever $D=\R^d$. We are now in position to state our main result.

\begin{theorem}[Sharp phase transition]\label{thm:sharpness} If $f$ satisfies Assumption~\ref{assump:q} for some $\beta>d$, then the following holds.
	\begin{itemize}
		\item For every $\ell<\ell_c$, there exists $c=c(\ell)>0$ such that for every $R\geq2$,
		\begin{equation}\label{eq:exp_decay}
			\P[\lr{}{\pazocal{E}(\ell)}{B_1}{B_R^{\mathsf{c}}}]\leq e^{-cR},
		\end{equation}
		\item For every $\ell>\ell_c$, there exists $M=M(\ell)>0$ such that $\pazocal{E}(\ell)$ percolates in $\R^2\times[0,M]^{d-2}$. Furthermore, there exists $c=c(\ell)>0$ such that for every $R\geq1$,
		\begin{equation}\label{eq:exp_decay2}
			\P[\nlr{}{\pazocal{E}(\ell)}{B_R}{\infty}]\leq e^{-cR^{d-1}}.
		\end{equation}
	\end{itemize}
\end{theorem}

\begin{remark}
It is known that there is a non-trivial phase transition -- actually, $\ell_c\in(-\infty,0]$ -- under certain general conditions \cite{MoSt83a,MoSt83b}, which are implied by the hypothesis of Theorem~\ref{thm:sharpness}. Furthermore, it is straightforward to deduce from Theorem~\ref{thm:sharpness} that, under our assumptions, one has $\ell_c=0$ for $d=2$. As mentioned above, this fact was already known -- even under weaker conditions \cite{MRV20} -- but our result provides an alternative proof.
\end{remark}

Let us mention some examples of fields satisfying the hypothesis of Theorem~\ref{thm:sharpness}. Probably the most important example is the so called \emph{Bargmann--Fock field}, which is characterized by the covariance kernel
$$\kappa(x)=e^{-\frac12 |x|^2}.$$
Indeed, in this case $\kappa= q\star q$ with $q(x)=(\frac{2}{\pi})^{\frac{d}{4}}e^{-|x|^2}$, and $q$ clearly verifies Assumption~\ref{assump:q} for every $\beta$ (in particular $\beta>d$). For this field, sharpness of phase transition was previously known only in the case $d=2$ \cite{RVb}. 

There are also examples with slower decay of correlations. For instance, consider the \emph{rational quadratic} kernel 
$$RQ_{\beta}(x):= (1+|x|^2)^{-\frac{\beta}{2}},$$
where $\beta>d/2$. By setting $q=RQ_{\beta}$ and $\kappa=q\star q$, it is clear that Assumption~\ref{assump:q} is verified (for the same value of $\beta$). By Theorem~\ref{thm:sharpness}, sharpness of phase transition holds for the associated Gaussian whenever $\beta>d$. One can also easily check that, for $|x|\geq1$, 
\begin{equation}\label{eq:decay_kappa}
	\kappa(x)\asymp
	\begin{cases} 
		|x|^{-2\beta+d}, &  \text{if $\tfrac{d}{2}<\beta< d$},\\
		|x|^{-d}\log|x|, &  \text{if $\beta=d$},\\
		|x|^{-\beta}, &  \text{if $\beta>d$}.
	\end{cases}
\end{equation}

We now discuss our assumptions, which are very similar to those considered in \cite{MV} for $d=2$. Although our approach is completely different from that of \cite{MV}, the reasons why we consider these assumptions are essentially the same. Property (i) of Assumption~\ref{assump:q} implies that the field $f$ can be represented as the convolution of $q$ with the standard white noise $W$ on $\R^d$:
\begin{equation}
	f=q\star W.
\end{equation}
This representation allows us to easily define a finite-range version of the model by simply truncating $q$.
Properties (ii) and (iii), respectively, guarantee that $f$ is \emph{locally} well approximated by a truncated and discretized version of $f$. Furthermore, the greater $\beta$ is, the better such a local approximation becomes, see Proposition~\ref{prop:local_comp}. 
Property (iv) implies that $f$ has $\Z^d$-symmetries, which allows us to directly import classical results and techniques from discrete percolation theory. 
Finally, property (v) implies that $f$ is an increasing function of $W$, thus allowing us to make use of the FKG inequality.

\subsection{Global comparison} Our strategy to prove Theorem~\ref{thm:sharpness} will consist in comparing $\pazocal{E}(\ell)$ with a sequence of models for which sharpness is known to be true. For that purpose, we will construct a truncated, discretized and noised version of $\pazocal{E}(\ell)$. As mentioned above, this approach is inspired by \cite{DGRS19}.

Let $\chi:\R^d\to [0,1]$ be a smooth and isotropic function satisfying
\begin{align*}
	\chi(x)= \begin{cases}
					1, &\text{if } |x|\leq 1/4 \\
					0, &\text{if } |x|\geq 1/2.		
				\end{cases}
\end{align*}
For $N\geq 1$, let $\chi_{\sN}(x)=\chi(x/N)$ and consider the truncated field
\begin{equation}
	f_{N}:=(q\chi_{\sN})\star W.
\end{equation}
Notice that $f_N$ has range of dependence $N$. 
Given $\varepsilon>0$, let $f_N^\varepsilon$ be the restriction of $f_N$ to $\varepsilon\Z^d$. One can also regard $f_N^{\varepsilon}$ as a Gaussian field on $\R^d$ by setting 
\begin{equation}
	f_N^\varepsilon(y):=f_N(x) ~~~\forall\, y\in x+[-\tfrac{\varepsilon}{2},\tfrac{\varepsilon}{2})^d,~ x\in \varepsilon\Z^d.
\end{equation}
Given $\delta\geq0$ (in addition to the discretization parameter $\varepsilon$), let $T_\delta=(T_\delta(x))_{x\in\varepsilon\Z^d}$ be an i.i.d. family of random variables independent of $W$ with distribution given by
\begin{align} 
	\begin{split}
	&\P[T_\delta(x)=0]=1-\delta ~~~ \text{and}\\ &\P[T_\delta(x)=-\infty]=\P[T_\delta(x)=+\infty]=\delta/2.
	\end{split}
\end{align}
We naturally extend $T_\delta$ to $\R^d$ by setting $T_\delta(y):=T_\delta(x)$ for every $y\in x+[-\frac{\varepsilon}{2},\frac{\varepsilon}{2})^d$, $x\in\varepsilon\Z^d$. Finally, for every $\ell\in \R$ we consider the percolation model
\begin{equation}
	\pazocal{E}_N^{\varepsilon,\delta}(\ell):= \{y\in \R^d :~ f_N^\varepsilon(y)+T_\delta(y)\geq -\ell\}.
\end{equation}
Notice that although $\pazocal{E}_N^{\varepsilon,\delta}(\ell)$ is defined in the continuum, it can be simply seen as a discrete site percolation model on $\varepsilon\Z^d$. In words, $\pazocal{E}_N^{\varepsilon,\delta}(\ell)$ is obtained by considering the excursions of $f_N$ above $-\ell$ on $\varepsilon\Z^d$ and then independently declaring each site $x\in\varepsilon\Z^d$ to be either unchanged (with probability $1-\delta$) or re-sampled as open or closed (with probability $\delta/2$ each). Let $\ell_c(N,\varepsilon,\delta)$ be the percolation critical level corresponding to $\pazocal{E}^{\varepsilon,\delta}_N$, as defined in \eqref{eq:l_c_def}. 

For every fixed $N\geq1$ and $\varepsilon,\delta>0$, one can easily check that the (discrete) percolation model $\pazocal{E}_{N}^{\varepsilon,\delta}(\ell)$ has a bounded-range i.i.d encoding, $\Z^d$-symmetries, positive association, finite energy, and a convenient sprinkling property -- the last two being the only reason why we introduce the noise. These properties are known to imply the following analogue of Theorem~\ref{thm:sharpness}. We refer the reader to Section 6 of \cite{DGRS19} for the precise definition of these properties and the proof of Theorem~\ref{thm:sharp_finite-range} below, which is based on adapting the techniques developed in \cite{DumRaoTas17b} and \cite{GriMar90}.

\begin{theorem}\label{thm:sharp_finite-range} Suppose that Assumption~\ref{assump:q} holds (actually, conditions (ii) and (iii) are not necessary). 
Then for every $N\geq1$ and $\varepsilon,\delta>0$, the following holds.
	\begin{itemize}
		\item For every $\ell<\ell_c(N,\varepsilon,\delta)$, there exists $c=c(N,\varepsilon,\delta,\ell)>0$ such that for every $R\geq2$,
		\begin{equation}\label{eq:exp_decay_truncated}
			\P[\lr{}{\pazocal{E}_N^{\varepsilon,\delta}(\ell)}{B_1}{B_R^{\mathsf{c}}}]\leq e^{-cR},
		\end{equation}
		\item For every $\ell>\ell_c(N,\varepsilon,\delta)$, there exists $M=M(N,\varepsilon,\delta,\ell)>0$ such that $\pazocal{E}_N^{\varepsilon,\delta}(\ell)$ percolates in $\R^2\times[0,M]^{d-2}$.Furthermore, there exists $c=c(N,\varepsilon,\delta,\ell)>0$ such that for every $R\geq1$,
		\begin{equation}\label{eq:exp_decay2_truncated}
			\P[\nlr{}{\pazocal{E}_N^{\varepsilon,\delta}(\ell)}{B_R}{\infty}]\leq e^{-cR^{d-1}}.
		\end{equation}
	\end{itemize}
\end{theorem}

We shall compare the probability of certain \emph{``admissible events''} for $\pazocal{E}(\ell)$ and $\pazocal{E}_N^{\varepsilon,\delta}(\ell\pm s)$, where $s>0$ is a small \emph{``sprinkling''} parameter, thus allowing to transfer Theorem~\ref{thm:sharp_finite-range} from $\pazocal{E}_{N}^{\varepsilon,\delta}(\ell)$ to $\pazocal{E}(\ell$). For this purpose, we introduce the following.

\begin{definition}[Admissible events]\label{def:admissible}
We say that an event $A$ (on the space of subsets of $\R^d$) is admissible if it has the form 
$$A=\{\lr{D}{}{S_1}{S_2^{\mathsf{c}}}\},$$ 
with either $S_1=[-r,r]^d$, $S_2=[-R,R]^d$ and $D=\R^d$ or $S_1=[-r,r]^2\times[-M,M]^{d-2}$, $S_2=[-R,R]^2\times[-M,M]^{d-2}$ and $D=\R^2\times [-M,M]^{d-2}$, where $M\geq0$ and $R\geq r\geq0$. Given a random subset $\pazocal{L}\subseteq\R^d$ and an event $A$ as above, we may write $\pazocal{L}\in A$ in place of $\{\lr{D}{\pazocal{L}}{S_1}{S_2^{\mathsf{c}}}\}$.
\end{definition}

We are now in position to state the following global comparison theorem, which is the heart of our strategy.  Theorem~\ref{thm:sharpness}, whose proof is shortly presented below, is a straightforward consequence of Theorems~\ref{thm:sharp_finite-range} and \ref{thm:comparison}. We believe that Theorem~\ref{thm:comparison} might have other applications though.

\begin{theorem}[Global comparison]\label{thm:comparison}
	Suppose that Assumption~\ref{assump:q} holds for some $\beta>d$. Then for every $\ell_0>0$ and $s>0$, there exist $N\geq1$ and $\varepsilon,\delta>0$ such that for every $\ell\in [-\ell_0,\ell_0]$ and every admissible event $A$,
	\begin{equation}\label{eq:comparison}
		\P[\pazocal{E}_N^{\varepsilon,\delta}(\ell-s)\in A]\leq \P[\pazocal{E}(\ell)\in A]\leq \P[\pazocal{E}_N^{\varepsilon,\delta}(\ell+s)\in A].
	\end{equation}
\end{theorem}

\begin{remark}\label{rmk:quantitative}
	The proof of Theorem~\ref{thm:comparison} actually gives a more quantitative statement: for any $\eta\in(0,\beta-d)$, one can take $s\asymp N^{-\eta}$, $\varepsilon=N^{-\beta+\tfrac{d}{2}}$ and $\delta=e^{-N^{2\beta-d-2\eta}}$ -- see \eqref{eq:choice_parameters}. If one assumes a super-polynomial decay for $\partial^\alpha q$, $|\alpha|\leq1$, then it is straightforward to obtain a corresponding better bound for $s$ in terms of $N$. In particular, for the Bargmann--Fock field one can take $N\asymp \sqrt{\log s^{-1}}$.
\end{remark}

\begin{remark} \label{rmk:delta=0}
One can easily check that the proof of Theorem~\ref{thm:comparison} still works if one takes $\varepsilon$ and $\delta$ to be even smaller than those referred to in the previous remark. In particular, one can can take $\varepsilon=\delta=0$, i.e.~compare $\{f\geq -\ell\}$ with $\{f_N\geq -(\ell\pm s)\}$.
One can probably also adapt the proof to other reasonable notions of truncation and discretization as well, such as those considered in \cite{MV} for instance.
\end{remark}

\begin{remark}\label{rmk:paper_vs_GFF} 
	As we have already mentioned, Theorem~\ref{thm:comparison} differs fundamentally from the comparison obtained in Proposition 1.4 of \cite{DGRS19} as the former does \textit{not} require the assumption that $\ell$ lies in a certain \emph{fictitious regime}, and is therefore a non-vacuous statement. Furthermore, there is no additive error term in \eqref{eq:comparison}. 
\end{remark}

\begin{remark}\label{rmk:admissible_events}
	The class of events for which we can prove \eqref{eq:comparison} is actually larger than those introduced in Definition~\ref{def:admissible}. Roughly speaking, we only need that the sets $S_1, S_2$ and $D$ to be such that there exists a ``uniformly flat surface'' in $D$ separating $S_1$ from $S_2$ -- see Remark~\ref{rmk:beta>d}. In order to avoid technicalities, we decided to state the results only for box-crossings in the full space and in 2D slabs as those events are sufficient to deduce Theorem~\ref{thm:sharpness}.
\end{remark}

We now explain how to derive Theorem~\ref{thm:sharpness} from Theorems~\ref{thm:sharp_finite-range} and \ref{thm:comparison}. 

\begin{proof}[Proof of Theorem~\ref{thm:sharpness}]
 	Fix $\ell_0>0$ a real number such that $\ell_c\in(-\ell_0,\ell_0)$. First, notice that by applying \eqref{eq:comparison} with the admissible events $A_R=\{\lr{}{}{B_1}{B_R^{\mathsf{c}}}\}$ and letting $R\to\infty$, one readily concludes that $|\ell_c-\ell_c(N,\varepsilon,\delta)|\leq s$. 
 	
 	Let $\ell<\ell_c$ and assume without loss of generality that $\ell>-\ell_0$. Let $s=\frac{\ell_c-\ell}{3}$ and $N,\varepsilon,\delta$ given by Theorem~\ref{thm:comparison}. Since $\ell+s<\ell_c-s\leq \ell_c(N,\varepsilon,\delta)$, the desired exponential decay \eqref{eq:exp_decay} follows directly by applying \eqref{eq:comparison} for $A_R$ combined with \eqref{eq:exp_decay_truncated}.
 	
 	Let $\ell>\ell_c$ and assume without loss of generality that $\ell<\ell_0$. Let $s=\frac{\ell-\ell_c}{3}$ and $N,\varepsilon,\delta$ given by Theorem~\ref{thm:comparison}. Since $\ell-s>\ell_c+s\geq \ell_c(N,\varepsilon,\delta)$, we know by Theorem~\ref{thm:sharp_finite-range} that there exists $M>0$ such that $\pazocal{E}_N^{\varepsilon,\delta}(\ell-s)$ percolates in $\R^2\times[-M,M]^{d-2}$ and satisfies \eqref{eq:exp_decay2_truncated}. By applying \eqref{eq:comparison} with the admissible events $A_R=\big\{\lr{\R^2\times{[}-M,M{]}^{d-2}}{}{[-1,1]^2\times[-M,M]^{d-2}}{([-R,R]^2)^{\mathsf{c}}\times[-M,M]^{d-2}}\big\}$ and letting $R\to\infty$, one directly concludes that $\pazocal{E}(\ell)$ also percolates in $\R^2\times[-M,M]^{d-2}$. The bound \eqref{eq:exp_decay2} follows by applying \eqref{eq:comparison} for $A_{R,R'}=\{\lr{}{}{B_R}{B_{R'}^{\mathsf{c}}}\}$, taking $R'\to\infty$ and using \eqref{eq:exp_decay2_truncated}.
\end{proof}

\begin{remark}\label{rmk:supercritical}
Notice that Theorem~\ref{thm:comparison} implies more than Theorem~\ref{thm:sharpness}. One can deduce that the asymptotic behavior of any admissible event is the same as for the truncated field, which in turn can be proved to be similar to Bernoulli percolation. We also hope that the quantitative relation between $s$ and $N$ mentioned in Remark~\ref{rmk:quantitative} could be used to transfer near-critical results from $\pazocal{E}^{\varepsilon,\delta}_N$ to $\pazocal{E}$.
\end{remark}
\subsection{Open questions}

We now discus some questions left open by the present work. First, we highlight a limitation of the version of ``supercritical sharpness'' we prove here, i.e.~the existence of an infinite cluster on thick slabs. For application purposes, the most useful version of supercritical sharpness would be the fact that a certain ``local uniqueness'' event has probability very close to $1$ in the supercritical phase. This is usually the starting point of any renormalization argument aimed at studying the geometry of infinite cluster or the large deviation behavior of the finite clusters, see e.g.~\cite{Cer00,Bar04,MR3650417}. In Bernoulli percolation, it is not hard to obtain local uniqueness from percolation on slabs -- see e.g.~\cite{Gri99a}. For general Gaussian fields though, it is not clear how one would prove this implication. Nonetheless, inspired on ideas from \cite{benjaminitassion17}, it is proved in Section~4 of \cite{DGRS19} that for the discrete GFF local uniqueness happens at level $\ell+\delta$ with high probability whenever $\P[\lr{}{\pazocal{E}(\ell)}{B_{u(R)}}{B_R^{\mathsf{c}}}]\geq 1-o(R^{-d})$, where $u(R)=\exp[(\log R)^3]$. Since Theorem~\ref{thm:sharpness} provides an even stronger bound \eqref{eq:exp_decay2} in the whole supercritical phase, one can expect that an adaptation of the techniques from \cite{DGRS19} may lead to a proof of local uniqueness for the Gaussian fields considered here. However, the aforementioned lack of finite-range decomposition and finite-energy property for general Gaussian fields make the implementation of such an argument considerably harder.

It would be especially interesting to prove a similar result for the so called \emph{monochromatic random wave}, which is one of the most relevant examples of continuous Gaussian fields. Studying this field is a much more challenging task as it does not admit a white noise representation and its covariance kernel oscillates and has a quite slow decay. However, great progress in the case $d=2$ has been made recently \cite{MRV20}.

We also expect that sharpness of phase transition holds under Assumption~\ref{assump:q} for every $\beta>d/2$. However, the decay of connection and disconnection events are not always the same as in  \eqref{eq:exp_decay} and \eqref{eq:exp_decay2}. Indeed, we conjecture that under Assumption~\ref{assump:q} the correct decay of connection events in the subcritical regime $\ell<\ell_c$ is, in general, the following:
\begin{equation}\label{eq:decay_onearm}
	-\log\P[\lr{}{\pazocal{E}(\ell)}{B_1}{B_R^{\mathsf{c}}}]\asymp
	\begin{cases} 
		R^{2\beta-d}, &  \text{if $\tfrac{d}{2}<\beta< \tfrac{d+1}{2}$},\\
		{R}/{\log R}, &  \text{if $\beta=\tfrac{d+1}{2}$},\\
		R, &  \text{if $\beta>\tfrac{d+1}{2}$}.
	\end{cases}
\end{equation}
As for the correct decay of disconnection events in the supercritical regime $\ell>\ell_c$, we conjecture that  under Assumption~\ref{assump:q}, one should have the following:
\begin{equation}\label{eq:decay_discon}
	-\log\P[\nlr{}{\pazocal{E}(\ell)}{B_R}{\infty}]\asymp
	\begin{cases} 
		R^{2\beta-d}, &  \text{if $\tfrac{d}{2}<\beta<d-\tfrac{1}{2}$},\\
		{R^{d-1}}/{\log R}, &  \text{if $\beta=d-\tfrac{1}{2}$},\\
		R^{d-1}, &  \text{if $\beta>d-\tfrac{1}{2}$}.
	\end{cases}
\end{equation}
Let us mention that both decays \eqref{eq:decay_onearm} and \eqref{eq:decay_discon} have been proved for the discrete GFF on certain graphs \cite{PopovRath15,GRS21,DPR21,MR3417515,nitzschner2017solidification,nitzschner2018}.

We now discuss some aspects of our global comparison theorem. First, we do not know whether the comparison \eqref{eq:comp_N_2N} should still hold for arbitrary increasing event $A$ under the same hypothesis of Theorem~\ref{thm:comparison} (i.e.~$\beta>d$), thus leading to a stochastic domination -- see Remark~\ref{rmk:beta>d} for more details on this. In Remark~\ref{rmk:beta>d-1/2}, we mention a possible way to prove Theorem~\ref{thm:comparison} under the weaker assumption $\beta>d-\tfrac{1}{2}$. This is optimal because due to \eqref{eq:exp_decay2_truncated} and \eqref{eq:decay_discon}, the comparison \eqref{eq:comparison} cannot hold for $\beta\leq d-\tfrac12$ in the supercritical regime. On the subcritical regime $\ell<\ell_c$ though, we still believe that \eqref{eq:comparison} holds whenever $\tfrac{d+1}{2}<\beta\leq d-\tfrac{1}{2}$. Because of \eqref{eq:exp_decay_truncated} and \eqref{eq:decay_onearm}, the comparison \eqref{eq:comparison} cannot hold as stated in either subcritical or supercritical regimes for $\tfrac{d}{2}<\beta\leq\tfrac{d+1}{2}$. Nonetheless, one can still hope to prove sharpness of phase transition for all $\beta>\tfrac{d}{2}$ by showing a global comparison on a certain fictitious regime, as originally done in \cite{DGRS19} for the GFF. In this case though, the comparison statement obtained becomes vacuous a posteriori.

\paragraph{Notation.}
For $A,B\subseteq\R^d$, we denote $A+B:=\{a+b:~a\in A,~b\in B\}$. Given $x\in \R^d$ and $R\geq 0$, let $B_R(x):= x + B_R$, where $B_R=[-R,R]^d$. For $S\subset \R^d$, we denote  its boundary as $\partial S:=\bar{S}\setminus S^o$.  
We write $c,c',C,C'$ for generic constants in $(0,\infty)$ which may depend implicitly on the field $f$, the dimension $d$ and can also change from line to line. Their dependence on other parameters will always be explicit. Numbered constants $c_0,c_1,C_0,C_1,...$ refer to constants that are used repeatedly in the text and are numbered according to their first appearance.

\section{Local comparison and Cameron--Martin formula}

In this section we overview a few basic properties of Gaussian fields. First, we show that the local difference between a field and its truncated and discretized version is small with very high probability. Then we review the Cameron--Martin formula describing the Radon--Nikodym derivative of a Gaussian field with respect to its shift by a deterministic function, which will be instrumental when implementing a certain finite-energy argument in Section~\ref{sec:interpolation}.

By applying standard Gaussian bounds, one can prove that $f_N^\varepsilon$ approximates $f$ very well locally (i.e.~in a ball of radius $1$). This is the content of the following proposition, which is analogous to Proposition 3.11 from \cite{MV}. Although \cite{MV} is concerned with the case $d=2$ and considers a slightly different notion of discretization, their proof can be easily adapted to our framework and we thus choose to omit the proof here.
\begin{proposition}[Local comparison]\label{prop:local_comp}
	Suppose that Assumption~\ref{assump:q} holds for some $\beta>\frac{d}{2}$ (actually, conditions (iv) and (v) are not necessary). Then there exists $c>0$ such that for every $N\geq1$ and $\varepsilon>0$,
	\begin{align}
		\label{eq:local_comp_r}
		\P\big[ \sup_{y\in B_1}|f(y)-f_N(y)| \geq s \big] &\leq \exp\big(-cs^2 N^{2\beta-d}\big), &&\text{for all } s\geq N^{-\beta+\frac{d}{2}},\\
		\label{eq:local_comp_epsilon}
		\P\big[ \sup_{y\in B_1}|f_N(y)-f_N^\varepsilon(y)|\geq s \big] &\leq \exp\big(-cs^2 \varepsilon^{-2}\big), &&\text{for all } s\geq\varepsilon.
	\end{align}
\end{proposition}

\begin{remark}
	Under a stronger assumption on the decay of $\partial^\alpha q$, $|\alpha|\leq1$, the bound \eqref{eq:local_comp_r} can be be strengthened correspondingly. For the Bargmann--Fock field in particular one can prove that
	\begin{align}\label{eq:local_comp_r2}
		&\P\big[ \sup_{x\in B_1}|f(x)-f_N(x)| \geq s \big] \leq \exp\big(-cs^2 e^{cN^2}\big), &\text{\emph{for all} } s\geq e^{-\tfrac{c}{2}N^2}.
	\end{align}
\end{remark}

We now introduce the Cameron--Martin space associated to a Gaussian field $g$ on an abstract space $X$ with covariance kernel $K(x,y):=\E[g(x)g(y)]$, $x,y\in X$. First, let $G$ be the Hilbert space of centered Gaussian random variables given by the closure in $L^2$ of the linear space spanned by $g$, i.e.~the set
\begin{equation}
	\sum_{i\in \N}a_i g(y_i), ~\text{ with } y_i\in X,~a_i\in \R,~ \sum_{i,j\in\N}a_i a_j K(y_i,y_j)<\infty.
\end{equation}
Then define the (injective) linear map $P: G\to \R^X$ given by
\begin{equation}
	\xi \mapsto P(\xi)(\cdot):=\langle\xi,f(\cdot)\rangle_G=\E[\xi g(\cdot)].
\end{equation}
The function space $H:=P(G)$, equipped with the inner product
\begin{equation}
	\langle h_1,h_2\rangle_H:=\langle P^{-1}(h_1),P^{-1}(h_2)\rangle_G=\E[P^{-1}(h_1)P^{-1}(h_2)]
\end{equation}
is a Hilbert space known as the Cameron--Martin space associated to $g$. By construction, $P$ defines an isometry between $G$ and $H$. Note that for any $y\in X$, the function $K(y,\cdot)$ is in $H$ and $P^{-1}(K(y,\cdot))=g(y)$. The following is a classical result -- see e.g.~Theorems 14.1 and 3.33 of \cite{janson_1997}.

\begin{theorem}[Cameron--Martin]\label{thm:Cameron-Martin}
	Let $g$ be a centered Gaussian field and $H$ its Cameron-Martin space. Then, for every $h\in H$, the Radon--Nikodym derivative between the law of $g+h$ and the law of $g$ is given by
	\begin{equation}
		\exp\big\{P^{-1}(h) - \tfrac{1}{2}\E[P^{-1}(h)^2]\big\}.
	\end{equation}
	In particular, for every $g$-measurable event $E$,
	\begin{equation}
		\P[g+h\in E]=\E\left[\exp\big\{P^{-1}(h) - \tfrac{1}{2}\E[P^{-1}(h)^2]\big\} \,1_{g\in E}\right].
	\end{equation}
\end{theorem}

\section{Interpolation scheme}\label{sec:interpolation}

In this section we prove Theorem~\ref{thm:comparison}. Fix an arbitrary $\eta\in(0,\beta-d)$ and let $\gamma:=2\beta-d-2\eta>d$. For $N\geq1$, we set
\begin{equation}\label{eq:choice_parameters}
	s_{\sN}:=N^{-\eta},~ \varepsilon_{\sN}:=N^{-\beta+\tfrac{d}{2}} \text{ and } \delta_{\sN}:=e^{-N^\gamma}.
\end{equation}
Theorem~\ref{thm:comparison} is a direct consequence of the following proposition.
\begin{proposition}\label{prop:comp_N_2N}
	For every $\ell_0>0$, there exist $N_0=N_0(\ell_0)\geq1$ such that the following holds. For every $N\geq N_0$, $\ell\in[-\ell_0,\ell_0]$ and every admissible event $A$, 
	\begin{equation}\label{eq:comp_N_2N}
				\P[\pazocal{E}_{N}^{\varepsilon_{\sN}, \delta_{\sN}}(\ell-s_{\sN})\in A]\leq \P[\pazocal{E}_{2N}^{\varepsilon_{\si{2N}}, \delta_{\si{2N}}}(\ell)\in A]\leq \P[\pazocal{E}_{N}^{\varepsilon_{\sN}, \delta_{\sN}}(\ell+s_{\sN})\in A].
	\end{equation}
\end{proposition}

\begin{proof}[Proof of Theorem~\ref{thm:comparison}]
	 By standard properties of Gaussian fields, one can prove that under our assumptions, every admissible event $A$ is a continuity event for $f$ in the uniform-on-compacts topology, and as a consequence $\lim_{N\to\infty} \P[\pazocal{E}_{N}^{\varepsilon_{\sN}, \delta_{\sN}}(\ell)\in A]=\P[\pazocal{E}(\ell)\in A]$. The theorem then follows by successively applying Proposition~\ref{prop:comp_N_2N} to $N=2^iN_0$, $i\geq n_0$, and noticing that $\sum_{i\geq n_0} s_\si{2^i N_0}\asymp (2^{n_0} N_0)^{-\eta}< s$ for $n_0=n_0(s)$ sufficiently large.
\end{proof}

\begin{proof}[Proof of Proposition~\ref{prop:comp_N_2N}] Fix $N\geq1$ and $\ell\in[-\ell_0,\ell_0]$. We will construct an interpolation between $\pazocal{E}_{2N}^{\varepsilon_{\si{2N}}, \delta_{\si{2N}}}(\ell)$ and $\pazocal{E}_{N}^{\varepsilon_{\sN}, \delta_{\sN}}(\ell\pm s_{\sN})$ as follows. Let $\{x_0,x_1,x_2,\dots\}$ be an arbitrary enumeration of $2N\Z^d$. We start with $\pazocal{E}_{2N}(\ell)$ and at the $n$-th step of our procedure, we will change the model on $B'_N(x_n):=x_n+[-N,N)^d$ (note that these boxes perfectly pave $\R^d$) from $\pazocal{E}_{2N}^{\varepsilon_{\si{2N}}, \delta_{\si{2N}}}$ to $\pazocal{E}^{\varepsilon_{\sN}, \delta_{\sN}}_{N}$ at the same time as slightly sprinkling the parameter $\ell$ everywhere. We do so in such a way that at the ``$\infty$-th step'' of our procedure, we end up with $\pazocal{E}^{\varepsilon_{\sN}, \delta_{\sN}}_{N}(\ell\pm s_{\sN})$.

We now describe the precise construction. For simplicity, we will focus on the interpolation between $\pazocal{E}_{2N}^{\varepsilon_{\si{2N}}, \delta_{\si{2N}}}(\ell)$ and $\pazocal{E}_{N}^{\varepsilon_{\sN}, \delta_{\sN}}(\ell+s_{\sN})$ only -- the construction for $\pazocal{E}_{N}^{\varepsilon_{\sN}, \delta_{\sN}}(\ell-s_{\sN})$ is analogous, see Remark~\ref{rmk:interpolation}. 
First, let $\tau:\R^d\to\R$ be the function given by
\begin{equation}\label{eq:tau}
	\tau(y):=c_d(1+|x|_\infty^{d+1})^{-1}, ~~~\forall\, y\in x+[-\tfrac12,\tfrac12)^d	,~ x\in \Z^d,
\end{equation}
where the constant $c_d>0$ is chosen so that $\int_{\R^d} \tau=\sum_{x\in\Z^d} c_d(1+|x|_\infty^{d+1})^{-1}=1/2$.
Consider the increasing sequence $(\tau_k)_{k\in\frac12 \N}$ of ``sprinkling functions'' recursively defined by $\tau_0=0$ and, for all $n\in\N$,
\begin{align}
	\label{eq:tau_k_1}
	\tau_{n+\frac12} &:= \tau_n + \tfrac{1}{2} 1_{B'_{N}(x_n)} s_{\sN},\\
	\label{eq:tau_k_2}
	\tau_{n+1} &:= \tau_{n+\frac12} + \tau\big(\tfrac{\cdot-x_n}{2N}\big)s_{\sN}.
\end{align}
Notice that by construction, the limit $t_{\infty}:=\lim_{n\to\infty} t_n$ is the constant function equal to $s_{\sN}$.
We are now in position to define the interpolation. For every $k\in\frac12\N$, consider
\begin{align}
	\label{eq:I_1}
	\pazocal{I}^1_k &:= \Big\{y\in\bigcup_{i=0}^{\lceil k\rceil-1} B'_N(x_i):~ f_N^{\varepsilon_{N}}(y)+T_{\delta_{N}}(y)\geq -(\ell+\tau_k(y))\Big\},\\
	\label{eq:I_2}
	\pazocal{I}^2_k &:=\Big\{y\in\bigcup_{i=\lceil k\rceil}^{\infty} B'_N(x_i):~ f_{2N}^{\varepsilon_{2N}}(y)+T_{\delta_{2N}}(y)\geq -(\ell+\tau_k(y))\Big\},
\end{align}
and finally define
\begin{equation}
	\label{eq:I}
	\pazocal{I}_k := \pazocal{I}^1_k\cup \pazocal{I}^2_k.
\end{equation}
The sequence $(\pazocal{I}_k)_{k\in\frac12 \N}$ is an interpolation between $\pazocal{E}_{2N}^{\varepsilon_{\si{2N}},\delta_{\si{2N}}}(\ell)$ and $\pazocal{E}_{N}^{\varepsilon_{\sN}, \delta_{\sN}}(\ell+s_{\sN})$.
Indeed, by construction $\pazocal{I}_0=\pazocal{I}_0^2=\pazocal{E}_{2N}^{\varepsilon_{\si{2N}},\delta_{\si{2N}}}(\ell)$ and, since $t_{\infty}\equiv s_{\sN}$, we have
\begin{equation}
	\pazocal{I}_\infty:=\lim_{k\to\infty} \pazocal{I}_k = \lim_{k\to\infty} \pazocal{I}_k^1 =  \pazocal{E}_{N}^{\varepsilon_{\sN}, \delta_{\sN}}(\ell+s_{\sN}).
\end{equation}

In words, for every $n\in\N$, we construct $\pazocal{I}_{n+\frac12}$ from $\pazocal{I}_n$ by changing the model in $B'_N(x_n)$ from $\pazocal{E}_{2N}^{\varepsilon_{\si{2N}}, \delta_{\si{2N}}}$ to $\pazocal{E}_{N}^{\varepsilon_{\sN}, \delta_{\sN}}$ and sprinkling the level parameter by $s_{\sN}/2$ on the same box only (see \eqref{eq:tau_k_1}); and we construct $\pazocal{I}_{n+1}$ from $\pazocal{I}_{n+\frac12}$ by simply sprinkling everywhere by an integrable function $\tau$ of the (renormalized) distance to the ``center'' $x_n$ -- see \eqref{eq:tau_k_2}. Notice that $\pazocal{I}_k$ is an increasing function of $W$ and $T:=(T_{\delta_N},T_{\delta_{2N}})$ for each $k\in\tfrac12\Z$. Also notice that this function is $N$-local, i.e.~$\pazocal{I}_k$ at any point $x$ only depends (increasingly) on the restriction of $W$ and $T$ to $B_N(x)$. In particular, $\pazocal{I}_k$ satisfies the FKG inequality and has range of dependence $2N$.

The crucial property of this construction is that it is ``almost increasing'' in $k$. First, we obviously have $\pazocal{I}_{n+\frac12}\subset \pazocal{I}_{n+1}$ almost surely for every $n\in\N$. Second, $\pazocal{I}_{n}\subset \pazocal{I}_{n+\frac12}$ holds with very high probability. Indeed, notice that by construction the following inclusion holds 
\begin{equation}\label{eq:inclusion_1} 
	\{\pazocal{I}_{n}\subset \pazocal{I}_{n+\frac12}\}\supset \{\sup_{y\in B'_N(x_n)} |f_N^{\varepsilon_N}(y)-f_{2N}^{\varepsilon_{2N}}(y)|\leq s_{\sN}/2\} \cap \{T_{\delta_{N}} = T_{\delta_{2N}} = 0 \text{ on } B'_N(x_n)\}.
\end{equation} 
Due to Proposition~\ref{prop:local_comp} and the choice of $\varepsilon_{\sN}$, $\delta_{\sN}$ and $s_{\sN}$ in \eqref{eq:choice_parameters}, one can easily lower bound the probability of the event in the right hand side of \eqref{eq:inclusion_1} by $1-e^{-cN^\gamma}$, thus
\begin{equation}
	\label{eq:n_n+1/2}	
	\P[\pazocal{I}_{n}\subset \pazocal{I}_{n+\frac12}]\geq 1-e^{-cN^\gamma}.
\end{equation}

\begin{remark}\label{rmk:interpolation}
	When adapting the above construction to interpolate between $\pazocal{E}_{2N}^{\varepsilon_{\si{2N}}, \delta_{\si{2N}}}(\ell)$ and $\pazocal{E}_{N}^{\varepsilon_{\sN},\delta_{\sN}}(\ell-s_{\sN})$, it is more convenient to start with $\pazocal{I}_0=\pazocal{E}_{N}^{\varepsilon_{\sN}, \delta_{\sN}}(\ell-s_{\sN})$ and end up with $\pazocal{I}_{\infty}=\pazocal{E}_{2N}^{\varepsilon_{\si{2N}}, \delta_{\si{2N}}}(\ell)$, so that the ``almost increasing'' property above still holds, and the proof presented below directly applies to this case as well.
\end{remark}

We claim that, if $N$ is large enough, then for every admissible event $A$,
\begin{equation}\label{eq:n_n+1}
	\P[\pazocal{I}_n\in A] \leq \P[\pazocal{I}_{n+1}\in A] ~~~ \text{for all } n\geq0.
\end{equation}
We now proceed with the proof of \eqref{eq:n_n+1}, which readily implies the desired inequality \eqref{eq:comp_N_2N}. In what follows, $N$ and $A$ are though as fixed and thus often omitted from the notation, but every estimate will be uniform on them. Given an admissible event $A$, we define
\begin{align}
	\label{eq:def_p_n}
	&p_n:=\P[\pazocal{I}_n\in A]-\P[\pazocal{I}_{n+\frac12}\in A]\\
	\label{eq:def_q_n}
	&q_n:=\P[\pazocal{I}_{n+1}\in A]-\P[\pazocal{I}_{n+\frac12}\in A] ~~~(\geq0).
\end{align}
Our goal is then to prove that, if $N$ is large enough (not depending on $A$), then $p_n\leq q_n$ for every $n\geq0$. 
Given $n\geq0$, $y\in\R^d$ and $L\geq0$, we define the \emph{coarse pivotality} event
\begin{equation}\label{eq:piv_def1}
	\text{Piv}^{n}_{y}(L):=\{\pazocal{I}_{n+\frac12}\cup B_{L}(y) \in A\}\cap \{\pazocal{I}_{n+\frac12}\setminus B_{L}(y) \notin A\}.
\end{equation}
Notice that $\text{Piv}^{n}_{y}(L)$ depends only on $\pazocal{I}_{n+\frac12}$ restricted to the complement of $B_L(y)$.
Since $\pazocal{I}_n$ and $\pazocal{I}_{n+\frac12}$ are identical on $B_{N}(x_n)^{\mathsf{c}}$ almost surely, the following inclusions hold 
\begin{align*} 
\{\pazocal{I}_n\in A,~ \pazocal{I}_{n+\frac12}\notin A\} &\subset \{\pazocal{I}_{n}\subset \pazocal{I}_{n+\frac12}\}^{\mathsf{c}} \cap \text{Piv}^{n}_{x_n}(N) \\ 
&\subset \{\pazocal{I}_{n}\subset \pazocal{I}_{n+\frac12}\}^{\mathsf{c}} \cap \text{Piv}^{n}_{x_n}(4N).
\end{align*}
The event $\{\pazocal{I}_{n}\subset \pazocal{I}_{n+\frac12}\}$, besides satisfying the bound \eqref{eq:n_n+1/2}, only depends on $W$ and $T$ restricted to $B_{2N}(x_n)$, which in turn is independent of $\pazocal{I}_{n+\frac12}$ restricted to $B_{4N}(x_n)^{\mathsf{c}}$. Combining these observations, we obtain
\begin{align}\label{eq:decoupling}
	\begin{split}
	p_n&\leq\P[\pazocal{I}_n\in A,~ \pazocal{I}_{n+\frac12}\notin A] \\
	 &\leq \P[\{\pazocal{I}_{n}\subset \pazocal{I}_{n+\frac12}\}^{\mathsf{c}}\cap\text{Piv}^{n}_{x_n}(4N)] \leq e^{-cN^\gamma} \P[\text{Piv}^{n}_{x_n}(4N)].
	\end{split}
\end{align}
For $n,j\geq0$, let 
\begin{equation}
	p_n(x_j):=\P[\text{Piv}^{n}_{x_j}(4N)].
\end{equation}
In view of \eqref{eq:decoupling}, it remains to show that if $N$ is large enough, then $p_n(x_n)\leq e^{cN^\gamma}q_n$ for every event $A$ and all $n\geq0$. In other words, we want to construct the \emph{``sprinkling pivotality''} event $\{\pazocal{I}_{n+1}\in A\}\cap \{\pazocal{I}_{n+\frac12}\notin A\}$ -- which has probability $q_n$, see \eqref{eq:def_q_n} -- out of the coarse pivotality event $\text{Piv}^n_{x_n}(4N)$ -- which has probability $p_n(x_n)$ -- by paying a sufficiently small multiplicative price. This fact is a straightforward consequence of the following lemma. Roughly speaking, it says that if coarse pivotality happens at a given site $x_j$, then one can either perform a \emph{``local surgery''} -- whose cost depends on the sprinkling function $\tau$ -- to construct sprinkling pivotality, or a local bad event -- which has a very small probability -- happens around $x_j$ and one further recovers a coarse pivotality event at some site $x_{j'}$ near $x_j$.

\begin{lemma}[Local surgery]\label{lem:conv_bound}
	There exists $\tilde{\gamma}\in(d,\gamma)$ and constants $N_1\geq1$ and $c,C\in(0,\infty)$ such that for $N\geq N_1$, every admissible event $A$ and every $n,j\geq0$, 
	\begin{equation}\label{eq:conv_bound}
		p_n(x_j)\leq e^{CN^{\tilde{\gamma}}}\, \tau\big(\tfrac{x_j-x_n}{2N}\big)^{-1} \,q_n ~+~ e^{-cN^{\tilde{\gamma}-1}} \sum_{x_{j'}\in B_{8N}(x_j)} p_n(x_{j'}).
	\end{equation}
\end{lemma}

Before proving Lemma~\ref{lem:conv_bound}, let us conclude the proof that $p_n\leq q_n$ for every $n\geq0$. Starting with $j=n$ and then successively applying Lemma~\ref{lem:conv_bound}, one can easily prove by induction that for every integer $T\geq1$,
\begin{align*}
	\begin{split}
	p_n(x_n)
	& \leq e^{CN^{\tilde{\gamma}}} \Big(\sum_{t=0}^{T-1} e^{-cN^{\tilde{\gamma}-1}t}\, |2N\Z^d\cap B_{8N}|^t\, |2N\Z^d\cap B_{8Nt}|\, c_d^{-1}\big(1+(4t)^{d+1}\big)\Big)\,q_n \\
	& +e^{-cN^{\tilde{\gamma}-1}T} |2N\Z^d\cap B_{8N}|^T \sum_{x_j\in B_{8NT}(x_n)} p_n(x_j).
	\end{split}
\end{align*}  
Noting that the sum inside the parenthesis above is convergent and that the second term vanishes when $T\to\infty$ (recall that $p_n(x_i)\leq1$), we obtain
\begin{equation*}
	p_n(x_n)\leq C'e^{CN^{\tilde{\gamma}}}\,q_n,
\end{equation*}
which combined with \eqref{eq:decoupling} and the fact that $\tilde{\gamma}<\gamma$, implies that for $N$ is large enough we have $p_n\leq q_n$ for every admissible event $A$ and all $n\geq0$, as we wanted to prove.
\end{proof}

It remains to prove Lemma~\ref{lem:conv_bound}, which is the technical heart of our proof.

\begin{proof}[Proof of Lemma~\ref{lem:conv_bound}]	
	Fix $N\geq1$ and $A:=\{\lr{D}{}{S_1}{S_2^{\mathsf{c}}}\}$ an admissible event -- recall Definition~\ref{def:admissible}. We stress that every estimate below will be uniform on $N$ and $A$. For $K\subset \R^d$, we will denote by $(W,T)_{|K}$ the restriction of $W$ and $T=(T_{\delta_N},T_{\delta_{2N}})$ to $K$. Given $n\geq0$, $L \geq 0$ and $y\in\R^d$, we define the \emph{closed pivotality} event
\begin{equation}
	\label{eq:cpiv_def}
	\text{CPiv}^n_y(L):=\{\pazocal{I}_{n+\frac12}\cup B_{L}(y) \in A\}\cap \{\pazocal{I}_{n+\frac12} \notin A\}.
\end{equation}
In words, $B_L(y)$ is called closed pivotal if it is pivotal but $A$ does not happen. The proof will be divided into four steps. \\
\vspace{0cm}

\textbf{Step 1:} \textit{From $4N$-pivotal to $8N$-closed-pivotal.}\\

\noindent
Let $n,j\geq0$. Consider the following event
\begin{equation*}
	\text{Piv}^n_{x_j}(8N,4N):=\{\pazocal{I}_{n+\frac12}\cup B_{8N}(x_j) \in A\}\cap \{\pazocal{I}_{n+\frac12}\setminus B_{4N}(x_j) \notin A\}.
\end{equation*}
By definition $\text{Piv}^n_{x_j}(4N)\subset \text{Piv}^n_{x_j}(8N,4N)$, hence
\begin{equation}\label{eq:inclusion_piv}
	p_n(x_j)=\P[\text{Piv}^n_{x_j}(4N)]\leq \P[\text{Piv}^n_{x_j}(8N,4N)]. 
\end{equation}
Now, consider the event 
\begin{equation*}
	F:=\{\pazocal{I}_{n+\frac12}\cap \partial B_{4N}(x_j)=\emptyset\} \cap \{\pazocal{I}_{n+\frac12}\cap \partial S_1 \cap B_{4N}(x_j)=\emptyset\}.
\end{equation*}
and notice that 
\begin{equation*}
\text{Piv}^n_{x_j}(8N,4N)\cap F \subset  \text{CPiv}^n_{x_j}(8N).
\end{equation*} 
Furthermore, both events $\text{Piv}^n_{x_j}(8N,4N)$ and $F$ are decreasing in $(W,T)_{|B_{6N}(x_j)}$ and $F$ is $(W,T)_{|B_{6N}(x_j)}$-measurable. Therefore, conditioning on $(W,T)_{|B_{6N}(x_j)^{\mathsf{c}}}$ and applying the FKG inequality for $(W,T)_{|B_{6N}(x_j)}$ gives 
\begin{equation}\label{eq:FKG_piv_F}
	\P[\text{CPiv}^n_{x_j}(8N)]\geq \P[\text{Piv}^n_{x_j}(8N,4N)\cap F] \geq \P[\text{Piv}^n_{x_j}(8N,4N)]\P[F].
\end{equation}
We will now prove that the following inequality holds
\begin{equation}\label{eq:bound_F}
	\P[F]\geq e^{-CN^{d-1}}.
\end{equation}
First observe that as $f$ is a non-degenerate, continuous, stationary Gaussian field, there exist constants $r_0>0$ and $c_0(\ell_0)>0$ such that $c_0<\P[f(y)>\ell~ \forall y\in B_{r_0}(x)]<1-c_0$ for every $\ell\in[-\ell_0,\ell_0]$ and $x\in\R^d$. Since $f_N$ converges locally to $f$, this property is also valid for $f_N$ with $N$ sufficiently large. Since $\pazocal{I}_k$ is made of level-sets of either $f^{\varepsilon_{N},\delta_{N}}_N$ or $f^{\varepsilon_{2N},\delta_{2N}}_{2N}$, one deduces that there exists $c_1=c_1(\ell_0)$ such that $\P[B_{r_0}(x)\cap \pazocal{I}_{k}=\emptyset]>c_1$ and $\P[B_{r_0}(x)\subset \pazocal{I}_{k}]>c_1$ for every $x\in \R^d$. One can thus cover $\partial B_{4N}(x_j)$ by at most $C(r_0) N^{d-1}$ balls of radius $r_0$ and apply the FKG inequality to obtain
\begin{equation}\label{eq:bound_F_1}
	\P[\pazocal{I}_{n+\frac12}\cap \partial B_{4N}(x_j)=\emptyset]>c_1^{C(r_0)N^{d-1}}.
\end{equation}
Now notice that, by specific geometry of the set $S_1$ -- recall Definition~\ref{def:admissible} -- one can easily cover $\partial S_1 \cap B_{4N}(x_j)$ with at most  $C(r_0)N^{d-1}$ balls of radius $r_0$. Therefore, by the same reasoning as the proof of \eqref{eq:bound_F_1}, we have
\begin{equation}\label{eq:bound_F_2}
	\P[\pazocal{I}_{n+\frac12}\cap \partial S_1 \cap B_{4N}(x_j)=\emptyset]>c_1^{C(r_0)N^{d-1}}.
\end{equation}
By another application of FKG, \eqref{eq:bound_F_1} and \eqref{eq:bound_F_2} imply \eqref{eq:bound_F}. Putting the inequalities \eqref{eq:inclusion_piv}, \eqref{eq:FKG_piv_F} and \eqref{eq:bound_F} together give
\begin{equation}\label{eq:4Npiv_to_8Ncpiv}
	p_n(x_j)\leq e^{CN^{d-1}}\P[\text{CPiv}^n_{x_j}(8N)].
\end{equation}
\vspace{0cm}

\textbf{Step 2:} \textit{From $8N$-closed-pivotal to $2N$-closed-pivotal.}\\

\noindent
First, let $\mathcal{C}^D_{S_1}$ and $\mathcal{C}^D_{S_2^{\mathsf{c}}}$ denote the (union of) connected components of $\pazocal{I}_{n+\frac12}\cap D$ intersecting $S_1$ and $S_2^{\mathsf{c}}$, respectively. Consider the events
\begin{align*}
	&E_1:= \{\mathcal{C}_{S_1}^D\cap S_2^{\mathsf{c}}=\emptyset\}\cap\{\mathcal{C}_{S_1}^D\cap B_{8N}(x_j)\neq\emptyset\}\\
	&E_2(y):=\{\mathcal{C}_{S_2^{\mathsf{c}}}^D\cap B_{1}(y)\neq\emptyset\},
\end{align*}
and notice that, for $E(y):=E_1\cap E_2(y)$, we have
$$\text{CPiv}^n_{x_j}(8N)\subset\bigcup_{y} E(y),$$
where the union above is taken over $y\in B_{8N}(x_j)\cap D\cap \Z^d$. By a union bound over all the $|B_{8N}(x_j)\cap D\cap \Z^d|\leq CN^d$ possibilities (recall \eqref{eq:choice_parameters}), we can find at least one such site $y_0$ satisfying
\begin{equation}\label{eq:step2_1}
	\P[E(y_0)]\geq  cN^{-d}\, \P[\text{CPiv}^n_{x_j}(8N)].
\end{equation}
For every $\pazocal{C}\subset D\subset \R^d$, consider the event
$$F(y_0,\pazocal{C}):=\{\lr{D\cap B_{8N}(x_j)}{\pazocal{I}_{n+\frac12}}{B_1(y_0)}{\pazocal{C}+B_{2N}}\}\cap\{B_1(y_0)\cap (\pazocal{C}+B_{2N})^{\mathsf{c}}\subset\pazocal{I}_{n+\frac12}\},$$
and notice that, for $F(y_0,\mathcal{C}_{S_1}^D):=\bigcup_{\pazocal{C}} \,\{\mathcal{C}_{S_1}^D=\pazocal{C}\}\cap F(y_0,\pazocal{C})$, we have
\begin{equation}\label{eq:inclusion_EcapF}
	E(y_0)\cap F(y_0,\mathcal{C}_{S_1}^D) ~\subset \bigcup_{x_{j'}\in B_{8N}(x_j)}~ \text{Piv}^n_{x_{j'}}(2N).
\end{equation}
By a similar reasoning as in the proof of \eqref{eq:bound_F} above, one can prove that for every $\pazocal{C}\subset D$ such that $\pazocal{C}\cap B_{8N}(x_j)\neq\emptyset$, we have
\begin{equation}\label{eq:bound_F(y_0,C)}
	\P[F(y_0,\pazocal{C})]\geq e^{-CN}.
\end{equation}
Since both events $F(y_0,\pazocal{C})$ and $E_2(y_0)$ are increasing in $(W,T)_{|(\pazocal{C}+B_N)^{\mathsf{c}}}$, we can apply the FKG inequality to deduce that 
\begin{align}\label{eq:step2_2}
\begin{split}
	\P[F(y_0,\pazocal{C})\cap E_2(y_0) \,\big|\, (W,T)_{|\pazocal{C}+B_N}]&\geq \P[F(y_0,\pazocal{C}) \,\big|\, (W,T)_{|\pazocal{C}+B_N}] \,\, \P[E_2(y_0) \,\big|\, (W,T)_{|\pazocal{C}+B_N}]\\
	&=\P[F(y_0,\pazocal{C})] \,\, \P[E_2(y_0) \,\big|\, (W,T)_{|\pazocal{C}+B_N}]\\
	&\geq e^{-CN} \,\P[E_2(y_0) \,\big|\, (W,T)_{|\pazocal{C}+B_N}],
\end{split}
\end{align}
for every $\pazocal{C}\subset D$ such that $\pazocal{C}\cap B_{8N}(x_j)\neq\emptyset$. In the second line of \eqref{eq:step2_2} we used that $F(y_0,\pazocal{C})$ is independent of $(W,T)_{|\pazocal{C}+B_N}$ and in the third line we used \eqref{eq:bound_F(y_0,C)}. Notice that for every $\pazocal{C}$, the event $\{\mathcal{C}_{S_1}^D=\pazocal{C}\}$ is measurable with respect to $(W,T)_{|\pazocal{C}+B_N}$. Therefore, multiplying both sides of \eqref{eq:step2_2} by the indicator function of $\{\mathcal{C}_{S_1}^D=\pazocal{C}\}$, taking expectations and then summing over all the (finitely many) $\pazocal{C}$ such that $\pazocal{C}\cap S_2^{\mathsf{c}}=\emptyset$ and $\pazocal{C}\cap B_{8N}(x_j)\neq \emptyset$, gives
\begin{equation}\label{eq:step2_3}
	\P[E(y_0)\cap F(y_0,\mathcal{C}_{S_1}^D)]\geq e^{-CN}\P[E(y_0)].
\end{equation}
Combining \eqref{eq:step2_1}, \eqref{eq:inclusion_EcapF} and \eqref{eq:step2_3}, we obtain
\begin{equation}\label{eq:8Ncpiv_to_2Ncpiv}
	\P[\text{CPiv}^n_{x_j}(8N)]\leq e^{C'N}\sum_{x_{j'}\in B_{8N}(x_j)} \P[\text{CPiv}^n_{x_{j'}}(2N)].
\end{equation}
\vspace{0cm}

\textbf{Step 3:} \textit{From $2N$-closed-pivotal to good-$2N$-closed-pivotal or bad-$4N$-pivotal.}\\

\noindent
Fix any $\tilde{\gamma}\in(d,\gamma)$ and $x_{j'}\in B_{8N}(x_j)$. Consider the following ``good event''
\begin{equation}\label{eq:G_def}
	G:=\left\{\begin{array}{c} |f_N^{\varepsilon_N}|,\, |f_{2N}^{\varepsilon_{2N}}| \leq N^{\frac{\tilde{\gamma}-1}{2}} \text{ on } B_{2N}(x_{j'})\\ 
	\text{and } T_{\delta_N}=T_{\delta_{2N}}=0 ~\,\text{ on } B_{2N}(x_{j'})\end{array}\right\}.
\end{equation}
By standard Gaussian bounds and the choice of $\delta_N$ (and $\delta_{2N}$) in \eqref{eq:choice_parameters}, one can easily see that 
\begin{equation}\label{eq:G_bound}
	\P[G]\geq 1-e^{-cN^{\tilde{\gamma}-1}}.
\end{equation}
Since $\text{CPiv}^n_{x_{j'}}(2N)\subset \text{Piv}^n_{x_{j'}}(4N)$ and $\text{Piv}^n_{x_{j'}}(4N)$ is independent of $G$, we have
\begin{align}\label{eq:closepiv_to_goodclosepiv}
	\begin{split}
	\P[\text{CPiv}^n_{x_{j'}}(2N)\cap G^{\mathsf{c}}]&\leq \P[\text{Piv}^n_{x_{j'}}(4N)\cap G^{\mathsf{c}}]\\
	&\leq e^{-cN^{\tilde{\gamma}-1}}\P[\text{Piv}^n_{x_{j'}}(4N)]=e^{-cN^{\tilde{\gamma}-1}}p_n(x_{j'}),
	\end{split}
\end{align}
and hence
\begin{equation}\label{eq:piv_to_goodclosepiv}
\P[\text{CPiv}^n_{x_{j'}}(2N)] \leq \P[\text{CPiv}^n_{x_{j'}}(2N)\cap G] + e^{-c N^{\tilde{\gamma}-1}} p_n(x_{j'}).
\end{equation}
\vspace{0cm}

\textbf{Step 4:} \textit{From good-$2N$-closed-pivotal to sprinkling-pivotal.}\\ 

\noindent
We are going to prove the following inequality
\begin{equation}\label{eq:goodclosepiv_to_b}
	\P[\text{CPiv}^n_{x_{j'}}(2N)\cap G]\leq e^{C'' N^{\tilde{\gamma}}}\tau\big(\tfrac{x_{j'}-x_n}{2N}\big)^{-1}\, q_n.
\end{equation}
Notice that by combining \eqref{eq:4Npiv_to_8Ncpiv}, \eqref{eq:8Ncpiv_to_2Ncpiv}, \eqref{eq:piv_to_goodclosepiv}, \eqref{eq:goodclosepiv_to_b} and reminding that $\tilde{\gamma}-1>d-1$, one readily obtains the desired bound \eqref{eq:conv_bound}.

We now explain how to prove \eqref{eq:goodclosepiv_to_b}. The idea is that the bound on $f_{N}^{\varepsilon_N}$ and $f_{2N}^{\varepsilon_{2N}}$ provided by the good event $G$ allows us to apply the Cameron--Martin theorem in order to shift the fields by an appropriate deterministic function, ultimately leading to the sprinkling-pivotality event $\{\pazocal{I}_{n+1}\in A\}\cap \{\pazocal{I}_{n+\frac12}\notin A\}$, which has probability $q_n$. 
First, consider the Gaussian field $g_{\sN}$ on $X:=\R^d\times\{1,2\}$ given by
$$g_{\sN}(y,i):=f_{iN}^{\varepsilon_{iN}}(y)$$
and let $K_N$ be the covariance kernel of $g_{\sN}$. Notice that for every $k\in\frac12 \N$, $\pazocal{I}_{k}$ is a function of $g_N$ (and $T$), and we will henceforth stress this by writing $\pazocal{I}_k=\pazocal{I}_k(g_\si{N})=\pazocal{I}_k(g_\si{N},T)$. In particular, since $\pazocal{I}_{n+1}$ is simply a sprinkling of $\pazocal{I}_{n+\frac12}$ by the function $\tau(\tfrac{\cdot-x_n}{2N})s_{\sN}$ (recall \eqref{eq:tau_k_2}--\eqref{eq:I}), we have \begin{equation}\label{eq:n+1/2_to_n+1}
\pazocal{I}_{n+1}(g_{\sN})=\pazocal{I}_{n+\frac12}\big(g_{\sN}+\tau(\tfrac{\cdot-x_n}{2N})s_{\sN}\big).
\end{equation}

By definition, $K_N((y,i),(y',i'))$ converges uniformly, as $N\to\infty$, to $\kappa(y,y')$ for every $y,y'\in\R^d$ and $i,i'\in\{1,2\}$. Since $\kappa$ is continuous and $\kappa(y,y)=\kappa(0,0)>0$, one can find $N_1>0$, $c_2>0$ and $r_1>0$ such that for every $N\geq N_1$,
\begin{equation}\label{eq:K_N_positive}
	K_N((y,i),(y',i'))\geq c_2 ~~~ \forall\, y,y'\in \R^d \text{ with } |y-y'|\leq r_1 \text{ and } \forall\, i,i'\in\{1,2\}.
\end{equation}
One can also easily check that $K_N$ has a decay analogous to that of $\kappa$ in \eqref{eq:decay_kappa}, i.e.~there exists $C>0$ such that for every $N>0$, 
\begin{equation}\label{eq:decay_K_N}
	K_N((y,i),(y',i'))\leq C|y-y'|^{-\beta}\, 1_{\{|y-y'|\leq 4N\}}~~~ \forall\, y,y'\in \R^d \text{ and } \forall\, i,i'\in\{1,2\}.
\end{equation}

Notice that, on the event $\text{CPiv}^n_{x_{j'}}(2N)$, one can find a pair of vertices $y_1, y_2 \in B_{2N}(x_{j'})\cap r_1\Z^d$ such that $\mathcal{C}_{S_1}^D\cap B_{r_1}(y_1)\neq\emptyset$ and $\mathcal{C}_{S_2^{\mathsf{c}}}^D\cap B_{r_1}(y_2)\neq\emptyset$. For any such pair, one can deterministically associate a sequence of points $\Gamma=\{z_1,z_2,...,z_m\}$ in $B_{2N}(x_{j'})$ satisfying $z_1=y_1$, $z_m=y_2$ and $|z_k-z_{k+1}|\leq r_1$ for every $k\in\{1,2,...,m-1\}$ with $m\leq C(r_1)N$. We consider the following function (which depends on $y_1$ and $y_2$)
\begin{equation}
	h:=\sum_{\substack{k\in\{1,2,...,m\} \\ i\in\{1,2\}}} K_N((z_k,i),\cdot)
\end{equation}
Notice that $h$ is in the Cameron--Martin space $H$ associated to $g_{\sN}$ and 
\begin{equation}\label{eq:P^-1(h)}
	P^{-1}(h)= \sum_{k=1}^{m} f_N^{\varepsilon_N}(z_k)+f_{2N}^{\varepsilon_{2N}}(z_k).
\end{equation}
It follows from \eqref{eq:K_N_positive}, \eqref{eq:decay_K_N} and $K_N\geq0$ that there exists $C_0,C_1\in(0,\infty)$ such that
\begin{enumerate}[(a)]
	\item $0\leq h(z,i)\leq C_0\, 1_{\{|z-x_{j'}|\leq 6N\}}~~ \forall\, z\in\R^d, ~\forall\, i\in\{1,2\}$,
	\item $h(z,i)\geq c_2 ~~ \forall\, z\in \bar{\Gamma}, ~\forall\, i\in\{1,2\}$, 
	\item $\E[P^{-1}(h)^2]=\sum_{\substack{k, k'\in\{1,2,...,m\} \\ i,i'\in\{1,2\}}} K_N((z_k,i),(z_{k'},i')) \leq C_1N$,
\end{enumerate}
where $\bar{\Gamma}:=\bigcup_{k=1}^{m} B_{r_1}(z_i)$ is the ``$r_1$-thickened path" determined by $\Gamma$.
Let 
$$\tilde{s}=\tilde{s}(N,n,x_{j'}):= C_0^{-1} \, \min_{y\in B_{6N}(x_{j'})} \tau\big(\tfrac{y-x_n}{2N}\big)s_{\sN},$$ 
which is defined in such a way that, in view of property (a) above, we have
\begin{equation}\label{eq:sh<sprinkling} 
\tilde{s}h\leq \tau(\tfrac{\cdot-x_n}{2N})s_{\sN}.
\end{equation}
Set $L:=(N^{\frac{\tilde{\gamma}-1}{2}}+\ell_0+s_{\sN})/c_2$. It follows from property (b) and \eqref{eq:G_def} that on the event $G$ we have $\bar{\Gamma}\subset \pazocal{I}_{n+\frac12}\big(g_{\sN} + L h\big)$. Therefore, on $\text{CPiv}^n_{x_{j'}}(2N)\cap G$, the event $A$ happens for $\pazocal{I}_{n+\frac12}\big(g_{\sN} + L h\big)$ but it does not happen for $\pazocal{I}_{n+\frac12}=\pazocal{I}_{n+\frac12}(g_{\sN})$. We then deduce by monotonicity that on $\text{CPiv}^n_{x_{j'}}(2N)\cap G$, there exist sites $y_1$ and $y_2$ (which determine $h$) as above and an integer $t\in[0 \,,\, L/\tilde{s})$ such that the event
\begin{equation}\label{eq:E_def}
	E(y_1,y_2,t):=\big\{\pazocal{I}_{n+\frac12}\big(g_N + t\tilde{s} h\big)\notin A\big\} \cap \big\{\pazocal{I}_{n+\frac12}\big(g_N + (t+1)\tilde{s} h\big)\in A\big\}
\end{equation}
happens. Overall, we have proved the following inclusion
\begin{equation}
	\text{CPiv}^n_{x_{j'}}(2N)\cap G \, \subset \, \bigcup_{y_1,y_2,t} E(y_1,y_2,t)\cap G.
\end{equation}
By a union bound, we conclude that there exist $y_1$ and $y_2$ and  $t$ as above such that
\begin{align}\label{eq:CPiv_to_Et}
	\begin{split}
	\P[E(y_1,y_2,t)\cap G]&\geq \big|B_{2N}(x_{j'})\cap r_1\Z^d\big|^{-2} \,\, \Big\lceil\frac{\tilde{s}}{L}\Big\rceil \,\, \P[\text{CPiv}^n_{x_{j'}}(2N)\cap G] \\
	& \geq c N^{-2d}\, \tau(\tfrac{x_{j'}-x_n}{2N})\, s_{\sN}  \, \P[\text{CPiv}^n_{x_{j'}}(2N)\cap G].
	\end{split}
\end{align}

We will now apply the Cameron--Martin theorem to the event $E(y_1,y_2,t)\cap G$. 
First notice that by \eqref{eq:P^-1(h)} and the property (c) above, the following bound holds almost surely on the event $G$
\begin{equation}\label{eq:R-N_bound}
	\left|P^{-1}(-t\tilde{s} h)-\frac12 \E[P^{-1}(-t\tilde{s} h)^2]\right| = \left|t\tilde{s}P^{-1}(h)+\frac12 (t\tilde{s})^2\E[P^{-1}(h)^2]\right|\leq C_3 N^{\tilde{\gamma}}.
\end{equation}
Since $g_{\sN}-t\tilde{s} h\in E(y_1,y_2,t) \iff g_N\in E(y_1,y_2,0)$, one can apply the Theorem~\ref{thm:Cameron-Martin} for the field $g_{\sN}$ (conditionally on $T$, which is independent from $g_{\sN}$) and the shift function $-t\tilde{s} h\in H$, which combined with \eqref{eq:R-N_bound} gives
\begin{equation}\label{eq:Et_to_E0}
	\P[E(y_1,y_2,0)]\geq e^{-C_3 N^{\tilde{\gamma}}} \P[E(y_1,y_2,t)\cap G].
\end{equation}
Due to \eqref{eq:n+1/2_to_n+1}, \eqref{eq:sh<sprinkling} and \eqref{eq:E_def}, we have $E(y_1,y_2,0)\subset \big\{\pazocal{I}_{n+\frac12}\notin A\} \cap \big\{\pazocal{I}_{n+1}\in A\big\}$, hence 
\begin{equation}\label{eq:E0_to_b}
	\P[E(y_1,y_2,0)]\leq q_n.
\end{equation}
Combining \eqref{eq:CPiv_to_Et}, \eqref{eq:Et_to_E0} and \eqref{eq:E0_to_b} gives the desired inequality \eqref{eq:goodclosepiv_to_b}, thus concluding the proof.
\end{proof}	

\begin{remark}\label{rmk:B-F_adapt}
	In order to obtain the quantitative relation $N\asymp\sqrt{\log s^{-1}}$ claimed in Remark~\ref{rmk:quantitative} for the case of the Bargmann--Fock field, it suffices to make the following simple adaptations in the choice of parameters. Fix any $\gamma>d$ and let $s_{\si{N}}=N^{\frac{\gamma}{2}}e^{-\frac{c}{2} N^2}$, $\varepsilon_{\si{N}}=e^{-\frac{c}{2} N^2}$ and $\delta_{\si{N}}=e^{-N^{\gamma}}$, where $c$ is given by \eqref{eq:local_comp_r2}. With these changes, the rest of the proof works as before.
\end{remark}

\begin{remark}\label{rmk:beta>d-1/2}
	If one wanted to prove Theorem~\ref{thm:comparison} under the (conjecturally optimal) assumption $\beta>d-\frac12$, one could try to make the following changes to the proof presented above. First, take $\eta\in(0,\beta-d+\frac12)$ instead, which implies that $\gamma:=2\beta-d-2\eta>d-1$. Now, for $\tilde{\gamma}\in (d-1,\gamma)$, redefine $G$ as (in addition to $T=0$, as before) the existence of a ``linear path'' of sites satisfying $|f^{\varepsilon_N}_N|,\, |f^{\varepsilon_{2N}}_{2N}| \leq N^{\frac{\tilde{\gamma}-1}{2}}$ between \emph{every} pair of macroscopic connected sets in $B_{4N}(x_{j'})$, instead of asking for a bound on the whole box. By using renormalization techniques, it should be possible to prove the improved bound $\P[G]\geq 1-e^{-cN^{\tilde{\gamma}}}$ instead of \eqref{eq:G_bound}. The rest of the argument then works similarly as before with the additional observation that, for typical $N$ and $x_{j'}$, both $\mathcal{C}_{S_1}^D$ and $\mathcal{C}_{S_2^{\mathsf{c}}}^D$ are macroscopic in $B_{4N}(x_{j'})$ on the event $\text{CPiv}^n_{x_{j'}}(2N)$. Since this adaptation is non-trivial and would probably result in a substantially more technical proof, we chose not to follow this strategy and we only claim the result under the stronger assumption $\beta>d$.
\end{remark}

\begin{remark}\label{rmk:beta>d}
	We now discuss the definition of admissible events $A$ for which the comparison \eqref{eq:comparison} holds. First notice that Steps 1, 2 and 4 relied on the fact that $A$ is a connection event, but the precise geometry of $S_1$, $S_2$ and $D$ has only been used to obtain the inequality \eqref{eq:bound_F_2} in Step 1 -- we simply used the fact that $\partial S_1$ can be localy covered at any scale $N$ by $O(N^{d-1})$ unit balls, which can be seen as a sort of ``uniformly flat curvature'' property. If we assume the stronger inequality $\beta>d+\tfrac12$, then the above proof of \eqref{eq:comparison} works for arbitrary connection events $A=\{\lr{D}{}{S_1}{S_2^{\mathsf{c}}}\}$, regardless of the geometry of $S_1$, $S_2$ and $D$. Indeed, it suffices to choose $\gamma>\tilde{\gamma}>d+1$ and redefine the event $F$ in Step 1 to be $F=\{\pazocal{I}_{n+\frac12}\cap B_{4N}(x_j)=\emptyset\}$, in which case we have $\P[F]\geq e^{-cN^d}$ instead of \eqref{eq:bound_F}, and the proof of Lemma~\ref{lem:conv_bound} still works as $\tilde{\gamma}-1>d$. Actually, if one combines this with the modified definition of good event proposed in Remark~\ref{rmk:beta>d-1/2} above, one might even be able to prove \eqref{eq:comparison} for arbitrary connection events $A$ under our original assumption $\beta>d$. Since in this case the geometry of the connection event $A$ does not matter, it seems natural to expect that \eqref{eq:comparison} holds for every increasing event $A$, which would imply the stochastic domination $\pazocal{E}^{\varepsilon,\delta}_{N}(\ell-s)\prec\pazocal{E}(\ell)\prec\pazocal{E}^{\varepsilon,\delta}_{N}(\ell+s)$. However, it is not clear how one could adapt the above proof for increasing events. In fact, if $A$ was an arbitrary increasing event,
	Step 4 could be adapted -- possibly under a stronger assumption on $\beta$ -- by simply constructing $h$ to be uniformly positive on the whole ball $B_{2N}(x_{j'})$ instead of a linear path. In Step 2 though, we strongly used the fact that $A$ is a connection event in order to reduce the scale of pivotality by conditioning on one cluster and making surgeries in the complement of its $2N$-neighborhood.
\end{remark}

\paragraph{Acknowledgements.} We would like to thank Hugo Duminil-Copin and Alejandro Rivera for inspiring discussions in the course of this project. We are also grateful to Stephen Muirhead and Pierre-Fran\c{c}ois Rodriguez for their valuable comments on an earlier draft of this paper. This research was supported by the Swiss National Science Foundation and the NCCR SwissMAP.


\end{document}